\documentclass[10pt]{amsart}
\usepackage{latexsym,amsmath,amssymb}
\usepackage{mathrsfs}
\usepackage{mathabx}
\usepackage{color}
\usepackage{indentfirst}
\renewcommand{\epsilon}{\varepsilon}

\newtheorem{theorem}{Theorem}[section]
\newtheorem{corollary}{Corollary}[section]
\newtheorem{lemma}{Lemma}[section]

\newtheorem{remark}{Remark}[section]

\newcommand{\bal}{\begin{align}}
	\newcommand{\bbal}{\begin{align*}}
		\newcommand{\beq}{\begin{equation}}
			\newcommand{\eeq}{\end{equation}}
		\newcommand{\bca}{\begin{cases}}
			\newcommand{\eca}{\end{cases}}
		
		\newcommand{\pa}{\partial}
		\newcommand{\fr}{\frac}
		\newcommand{\na}{\nabla}
		\newcommand{\De}{\Delta}

		\newcommand{\cd}{\cdot}
		
		\newcommand{\dd}{\mathrm{d}}

		\newcommand{\R}{\mathbb{R}}

		\title [Global Large solutions of the 2D supercritical SQG equations]{Global large, smooth solutions of the 2D surface quasi-geostrophic equations}
		\date{\today}
		
\begin{document}

			\maketitle

			\centerline{
				\author{
					Huali Zhang
					\footnote{School of Mathematics and Statistics, Changsha University of Science and Technology, Changsha, Hunan, 410114, P. R. China.
						{\it Email: hualizhang@csust.edu.cn}} }
\and
				  \quad Jinlu Li
				  \footnote{Corresponding Author: School of Mathematics and Information Science, Guangzhou University, Guangzhou 510006, P. R. China.
				 {\it Email: lijinlu@gnnu.edu.cn}
				 }
							}

			{\bf Abstract:} In this paper, we prove the global regularity of smooth solutions to 2D surface quasi-geostrophic (SQG) equations with super-critical dissipation for a class of large initial data, where the velocity and temperature can be arbitrarily large in spaces $L^\infty(\mathbb{R}^2)$ and $H^3(\mathbb{R}^2)$. This result can be seen as an improvement work of Liu-Pan-Wu \cite{Liu}, for it's without any smallness hypothesis of the $L^\infty(\mathbb{R}^2)$ norm of the initial data.
			
			{\bf Keywords:} SQG equation; supercritical dissipation; large initial data; global smooth solutions.
			
			{\bf MSC (2010):} 35K55; 86A10.
		%	\vskip0mm\noindent{\hrulefill }
			
			\section{Introduction}\label{sec1}
			In this paper we consider the Cauchy problem of surface quasi-geostrophic (SQG) equations with supercritical dissipation or damping. The SQG equation we are concerned with here assumes the form
			\begin{eqnarray}\label{sqg}
				\left\{\begin{array}{ll}
					\partial_t\theta+u\cd\na \theta +\mu (-\Delta)^{\alpha} \theta=0,& x\in \R^2,t>0,\\
					u=\na^{\perp}(-\Delta)^{\frac{1}{2}}\theta,& x\in \R^2,t> 0,\\
					\theta|_{t=0}=\theta_0,& x\in \R^2,\end{array}\right.
			\end{eqnarray}
			where the scalar function $\theta$ represents the potential temperature, $u$ is the fluid velocity and $\mu >0$, $\alpha \in [0,1]$, $\nabla^{\bot}=(\partial_{x_2}, -\partial_{x_1})^{\text{T}}$. The pseudo-differential
			operator $(-\Delta)^{\alpha}$ is denoted by
			\begin{equation*}
			\widehat{(-\Delta)^{\alpha}}=|\xi|^{2\alpha}\hat{f}(\xi),
			\end{equation*}
			where
			  $\hat{f}(\xi):=\int_{\mathbb{R}^2}e^{-i x \cdot \xi}f(x)dx$. The SQG equation \eqref{sqg} with $\alpha>\frac12, \alpha=\frac12, 0<\alpha<\frac12$ and $\alpha=0$ is called sub-critical SQG equation, critical SQG equation, supercritical SQG equation and damping SQG equation respectively.

			The SQG equation models the dynamics of the potential temperature $\theta$ of the 3D quasi-
			geostrophic equations on the 2D horizontal boundaries, and is useful in modeling
			atmospheric phenomena such as frontogenesis (see e.g. \cite{CMT}). Besides its wide application in physics, the SQG equation %(inviscid or dissipation)
			has been attracted numerous attention for its significance in the theory of partial differential equations, and the behavior of its strongly nonlinear solutions are strikingly analogous to that of the potentially singular solutions of the 3D incompressible Navier-Stokes and Euler equations. It is a subject of a large literature on the issue of global regularity solutions concerning the SQG equation, and much excellent work has been made on it. The difficulty of studying the global existence of solutions to the SQG equation crucially relies on the value of parameter $\alpha$.

			In the subcritical case $\alpha> \frac{1}{2}$, Constantin-Wu in \cite{CW} proved the existence of global in time smooth solutions. For the asymptotic behaviour of solutions of 2D SQG equation, please see the work \cite{CF} of Carrillo-Ferreira. In the critical case $\alpha=\frac{1}{2}$,
			the global regularity of solutions to the 2D SQG equation for any initial data has been successfully solved by Kiselev-Nazarov-Volberg and Caffarelli-Vasseur respectively (see in \cite{KNV, CV}). We also mention a result of maximum principle due to Cordoba and Cordoba \cite{CC} and a result of behavior in large time due to Schonbek and Schonbek \cite{CC}. In the supercritical case $0< \alpha< \frac{1}{2}$, Chae-Lee \cite{CL} proved the global existence of small solutions in the scaling invariant Besov space(see also Chen-Miao-Zhang \cite{CMZ}). There are also some important developments related to weak solutions and regularity criterion of the supercritical SQG equations, one could see the references \cite{D,K,S,DP,CW1} and so on. However, whether or not classical solutions to the supercritical dissipation SQG equation can develop finite time singularities remains an outstanding open problem. Very recently, Liu-Pan-Wu in \cite{Liu} established the global well-posedness of smooth solutions for the SQG equation with supercritical dissipation or damping for a class of large initial data. The initial data constructed in \cite{Liu} can be arbitrarily large in $H^s(\mathbb{R}^2), s>2-\alpha$
			%, while the $L^\infty(\mathbb{R}^2)$ norm of the intial data may be actually small
			. Motivated by their work, we hope to find a class of large data in the sense of $L^\infty(\mathbb{R}^2)$ such that the supercritical dissipation SQG equations still has a unique global solution. %and $H^3(\mathbb{R}^2)$

			The goal of this paper is to construct a class of large
			solutions for 2D SQG equation \eqref{sqg} with supercritical dissipation or damping. The striking highlight in our paper is that the initial data we constructed could be arbitrarily large in the space $L^\infty(\mathbb{R}^2)$(it's also large in the energy space $H^3$).  Based on the initial condition \eqref{condition}, we establish the global well-posedness of solutions to 2D SQG equation \eqref{sqg} with $\alpha \in [0,\frac{1}{2})$. Under Condition \eqref{condition}, we construct a class of initial data, which can be arbitrarily large in spaces $L^\infty(\mathbb{R}^2)$ and $H^3(\mathbb{R}^2)$. Compared to the work of Liu-Pan-Wu \cite{Liu}, we drop out the smallness hypothesis of $L^\infty(\mathbb{R}^2)$ norm of the initial data.

			The large solution of \eqref{sqg} we are going to construct is given by
			
			\begin{equation}\label{t3}
			\theta=\Theta+g,
			\end{equation}
			
			where $\Theta$ is a background solution of the linear equation
			\begin{equation}\label{t}
			\begin{split}
			&\partial_t\Theta+\mu (-\Delta)^{\alpha} \Theta=0, \ \ (t,x) \in \mathbb{R}^+ \times \mathbb{R}^2,
			\\
			& \Theta|_{t=0}=\Theta_0.
			\end{split}
			\end{equation}
			
			Our results are as follows.
		
		\begin{theorem}\label{thm}
			Consider the Cauchy's problem of the SQG equation \eqref{sqg}. Let $0\leq \alpha < \frac12$ and $\theta, \Theta$ be described as in \eqref{t3}-\eqref{t}, $U=(-\Delta)^{\frac{1}{2}}\nabla^\bot \Theta_0$. Assume that the initial data fulfills $\theta_0=\Theta_0+g_0$
			with
			\begin{eqnarray}\label{Equ1.2}
				\mathrm{supp} \ \hat{\Theta}_0(\xi) \subset\mathcal{C}:=\Big\{\xi \big| \  \frac43\leq |\xi|\leq \frac32\Big\} .
			\end{eqnarray}
			If there exists a sufficiently small positive constant $\epsilon \in (0, e^{-100}]$ and a universal constant $C$ such that
			\begin{align}\label{condition}
				\left(\|g_0\|^2_{H^3}+\int^{\infty}_{0}\|U\cd\na\Theta\|_{H^3}\dd t\right)\cd\exp\left(C\big(\int^{\infty}_{0}\|U, \Theta\|_{L^\infty}\dd t+\int^{\infty}_{0}\|U\cd\na\Theta\|_{H^3}\dd t\big)\right)\leq \epsilon.
			\end{align}
			Then the system \eqref{sqg} has a unique global solution $\theta=\Theta+g$, where $\Theta=e^{-(-\Delta)^{\alpha}t}\Theta_0$ and $g$ satisfies
			\begin{equation}\label{00}
			\|g(t)\|^2_{H^3}+ \frac{\mu}{2}\int^{t}_0 \| g(\tau,\cdot)\|^2_{H^3}d\tau \leq C \epsilon, \ \ \forall t \geq 0.
			\end{equation}
		for a universal constant $C(C>1)$.
		\end{theorem}
	\begin{remark}
	Throughout we use a notation $C$. It may be different from line to line, but it is a universal positive constant in this paper.
\end{remark}

Inspired by the work of Lei-Lin-Zhou\cite{LLZ} and Li-Yang-Yu \cite{Li}, we can construct a class of initial data, where the velocity and temperature can be arbitrarily large in spaces $L^\infty(\mathbb{R}^2)$ and $H^3(\mathbb{R}^2)$.
\begin{corollary}\label{cc}
	Let $\epsilon, g_0, \Theta_0, U_0$ be described in Theorem \ref{thm}. %For
	Choose small number $\delta$ such that $\epsilon=\delta^\frac{1}{2}|\log \delta|^{2}$. Take $$a_0(x_1,x_2)=\delta^{-1} (\log |\log \delta|)\cdot \chi(x_1,x_2)$$
	and
	\begin{equation}\label{c0}
	U_0=(-\Delta)^{\frac{1}{2}}\nabla^{\bot} a_0=((-\Delta)^{\frac{1}{2}}\partial_2 a_0, -(-\Delta)^{\frac{1}{2}}\partial_1 a_0)^{\text{T}}, \ \Theta_0=a_0,
	\end{equation}
	where the smooth function $\chi$ satisfies $\hat{\chi}(-\xi_1, -\xi_2)=\hat{\chi}(\xi_1, \xi_2)$, $\hat{\chi}(\xi)=1$ if $\xi \in \Big\{\xi \big| |\xi_1-\xi_2| \leq \frac{1}{3}\delta, \  \frac43\leq |\xi|\leq \frac32\Big\}$ and
	$\mathrm{supp} \ \hat{\chi}(\xi) \subset \Big\{\xi \big| |\xi_1-\xi_2| \leq \delta, \  \frac43\leq |\xi|\leq \frac32\Big\}$. If we assume $\|g_0\|^2_{H^3} \leq \delta^\frac{1}{2}$, then Condition \eqref{condition} holds. Moreover, we have
	\begin{equation}\label{c9}
	\begin{split}
	&\|U_0\|_{L^\infty}\geq  \frac{1}{5000}\log |\log \delta| , \ \ \|U_0\|_{L^2}\geq  \frac{1}{100} \delta^{-\frac{1}{2}}\log |\log \delta|,
	\\
	& \|\Theta_0\|_{L^\infty} \geq  \frac{1}{5000} \log |\log \delta|, \ \  \|\Theta_0\|_{L^2} \geq  \frac{1}{100} \delta^{-\frac{1}{2}}\log |\log \delta|.
	\end{split}
	\end{equation}
\end{corollary}
\begin{remark}
	In \cite{Liu}, we could calculate that the initial data constructed by Liu-Pan-Wu satisfies $\|\Theta_0\|_{L^\infty} \leq \delta^\frac{1}{2} |\log \delta|$. If we take it into the left hand side of our Condition \eqref{condition}, then it's equal to $\delta^{\frac{3}{2}-2\sigma} |\log \delta|  <\delta^\frac{1}{2}|\log \delta|^{2}= \epsilon$, for $\sigma \in (0, \frac{1}{2})$.
\end{remark}
\section{The proof of Theorem \ref{thm}}
In this section, we will prove Theorem \ref{thm}. Let us first introduce the commutator estimate.
%\begin{lemma}\label{L}\cite{CMT}
%	Let $s>0$. Let $p, p_2, p_3 \in(1,\infty)$ and $p_2,p_4 \in[1,\infty)$ satisfy
%	\begin{equation*}
%	\frac{1}{p}=\frac{1}{p_1}+\frac{1}{p_2}=\frac{1}{p_3}+\frac{1}{p_4}.
%	\end{equation*}
%	Then there exist two constants $C_1,C_2$,
%	\begin{equation*}
%	\begin{split}
%	&|| \Lambda^s(hr)||_{L^p} \leq C_1 \left( || \Lambda^s h||_{L^{p_1}} ||r||_{L^{p_2}}+ || \Lambda^{s}r||_{L^{p_3}} ||h||_{L^{p_4}} \right),
%	\\
%	&|| [\Lambda^s,h] r||_{L^p} \leq C_2 \left( || \Lambda^s h||_{L^{p_1}} ||r||_{L^{p_2}}+ || \Lambda^{s-1}r||_{L^{p_3}} ||\nabla h||_{L^{p_4}} \right).
%	\end{split}
%	\end{equation*}
%	Here $\Lambda=(-\Delta)^\frac{1}{2}$.
%\end{lemma}	
\begin{lemma}\label{lem}\cite{KP}
	Let $m$ be a positive integer,
	%=(m_1,m_2)$,  $m_1\geq 0, m_2 \geq
	 %$D^m=\partial_{x_1}^{m_1}\partial_{x_2}^{m_2},
	$h,f \in H^m(\mathbb{R}^2)$. The following commutator estimate
	\begin{equation}\label{200}
	\sum_{|\alpha| \leq m} \|D^\alpha(hf)-(D^\alpha h)f\|_{L^2} \leq C\left( \|h\|_{H^{m-1}}\|\nabla f\|_{L^\infty}+ \|h\|_{L^\infty}\|f\|_{H^m}\right)
	\end{equation}
	holds. Here $D^\alpha=\partial_{x_1}^{\alpha_1}\partial_{x_2}^{\alpha_2}, \alpha=(\alpha_1, \alpha_2)$.
\end{lemma}
We are ready to prove Theorem \ref{thm}.
\begin{proof}[Proof of Theorem \ref{thm}]
Recall $\Theta$ being a solution of \eqref{t}. In a result, $\Theta=e^{-(-\Delta)^\alpha t}\Theta_0$ and $\mathrm{supp} \ \hat{\Theta}(t,\xi) \subset\mathcal{C}, \forall t\geq 0$. Set $g=\theta-\Theta,v=u-U$. We then easily derive the equation of $g$ by substituting \eqref{t} in \eqref{sqg},
%\bbal
%\pa_tc+v\cd\na c+U\cd\na c+v\cd\na \Theta+(-\De)^{\alpha}c=U\cd\na \Theta
%\end{align*}
\begin{equation}\label{c}
\begin{split}
&\pa_tg+v\cd\na g+U\cd\na g+v\cd\na \Theta+\mu(-\De)^{\alpha}g=-U\cd\na \Theta,
\\
&\nabla \cdot g=0,
\\
& g|_{t=0}=g_0.
\end{split}
\end{equation}
Let us first give the $0$-order estimates. By direct calculation, we have
\begin{equation*}
\begin{split}
\fr12\frac{\dd}{\dd t}\|g\|^2_{L^2}+\mu\|\Lambda^{\alpha} g\|^2_{L^2}=&-\int_{\R^2} (v\cd\na g+U\cd\na g + v\cd\na \Theta+U\cd\na \Theta)gdx
\\
=& -\int_{\R^2}  (v\cd\na \Theta+U\cd\na \Theta)dx.
\end{split}
\end{equation*}
By H\"older's inequality and elliptic estimate, we have
\begin{equation}\label{10}
\begin{split}
\fr12\frac{\dd}{\dd t}\|g\|^2_{L^2}+\mu\|\Lambda^{\alpha} g\|^2_{L^2}
\leq & C  \|v\|_{L^2}\|g\|_{L^2} \| \nabla \Theta \|_{L^\infty}+\|g\|_{L^2}\|U\cd\na \Theta\|_{L^2}  \\
\leq & C  \|g\|^2_{L^2} \| \Theta \|_{L^\infty}+\|g\|_{L^2}\|U\cd\na \Theta\|_{L^2},
\end{split}
\end{equation}
where we use the Bernstein's inequality(for $\mathrm{supp} \ \hat{\Theta}(\xi)\subset\mathcal{C}$) and $\|v\|_{L^2}= \|\nabla^{\bot}(-\Delta)^{-\frac{1}{2}}g\|_{L^2}\leq C\|g\|_{L^2}$.

We are now in the stage to give the hihger-order estimates. Let $\beta=(\beta_1, \beta_2)^\mathrm{T}, \beta_1, \beta_2 \in \mathbb{N}^+$. Operating $D^\beta$ to \eqref{c} and taking the scalar product with $D^\beta g$, and then summing the result over $|\beta|=1,2,3$, we get
\begin{equation}\label{l}
\sum_{|\beta|=1,2,3}(\fr12\frac{\dd}{\dd t}\|g\|^2_{\dot{H}^\beta}+\mu\|\Lambda^{\alpha} g\|^2_{\dot{H}^\beta})\triangleq\sum^{4}_{i=1}I_i,
\end{equation}
where $\Lambda= (-\Delta)^\frac{1}{2}$ and
\bbal
&I_1=-\sum_{|\beta|=1,2,3}\int_{\R^2} D^\beta(v \cd \na g) \cd D^\beta g\dd x,
\\&I_2=-\sum_{|\beta|=1,2,3}\int_{\R^3}D^{\beta}(U\cd \na g)\cd D^\beta g\dd x,
\\&I_3=-\sum_{|\beta|=1,2,3}\int_{\R^3}D^{\beta}(v\cd \na \Theta)\cd D^{\beta}g\dd x,
\\&I_{4}=-\sum_{|\beta|=1,2,3}\int_{\R^3}D^{\beta}(U\cd\na\Theta)\cd D^{\beta}g\dd x.
\end{align*}

Using the fact that $\nabla \cdot v=\nabla \cdot U=0$, we could deduce that
\bbal
&I_1=-\sum_{|\beta|=1,2,3}\int_{\R^3}[D^{\beta},v\cd] \na g\cd D^\beta g\dd x,
\\&I_2=-\sum_{|\beta|=1,2,3}\int_{\R^3}[D^{\beta}, U\cdot]\na g \cdot D^\beta g\dd x,
\end{align*}
%\\&I_3=-\sum_{0\leq|\beta|\leq 3}\int_{\R^3}D^{\beta}(v\cd \na \Theta)\cd D^{\beta}c\dd x,
%\\&I_{4}=-\sum_{0\leq |\beta|\leq 3}\int_{\R^3}D^{\beta}(U\cd\na\Theta)\cd D^{\beta}c\dd x.

According to the commutate estimate in Lemma \ref{lem},
we obtain
\begin{equation*}
\begin{split}
|I_1|\leq&~\sum_{|\beta|=1,2,3}\|[D^{\beta},v\cd]\na g\|_{L^2}\|\na g\|_{H^2}
\\
\leq&~C\sum_{|\beta|=1,2,3} (\|\na v\|_{L^\infty}\|\nabla D^{\beta-1} g\|_{L^2}+\|D^{\beta} v\|_{L^2}\|\na g\|_{L^\infty})\|\na g\|_{H^2}.
\end{split}
\end{equation*}
Note $\alpha \in (0,\frac{1}{2})$ and $|\beta|=1,2,3$. By Gagliardo-Nirenberg's inequality, we have
\begin{equation*}
\begin{split}
&\|\nabla g\|_{L^\infty} \leq C \| \Lambda^\alpha g\|^{\frac{\alpha+1}{3}}_{L^2} \|\Lambda^{\alpha+3} g\|_{L^2}^{\frac{2-\alpha}{3}},
\\
&\|D^\beta g\|_{L^2} \leq C \| \Lambda^\alpha g\|^{\frac{\alpha}{3}}_{L^2} \|\Lambda^{\alpha+3} g\|_{L^2}^{1-\frac{\alpha}{3}}.
\end{split}
\end{equation*}
The similar estimates hold for $\|\nabla v\|_{L^\infty} $ and $\|D^\beta v\|_{L^2}$.

Therefore, we prove that
\begin{equation}\label{I1}
\begin{split}
|I_1|&\leq C  \sum_{|\beta|=1,2,3}\|\Lambda^\beta g\|_{L^2} \|\Lambda^\alpha g\|_{L^2}^{\frac{\alpha+1}{3}+\frac{\alpha}{3}} \|\Lambda^{\alpha+3} g\|_{L^2}^{\frac{2-\alpha}{3}+1-\frac{\alpha}{3}}
\\
& \leq C \sum_{|\beta|=1,2,3}\|\Lambda^\beta g\|_{L^2} \left( \|\Lambda^\alpha g\|^2_{L^2} +\|\Lambda^{\alpha+3} g\|^2_{L^2} \right)
\\
& \leq C \|g\|_{H^3} \|\Lambda^\alpha g\|^2_{H^3},
\end{split}
\end{equation}
where we use the fact that $\|v\|_{H^s}= \|\nabla^{\bot}(-\Delta)^{-\frac{1}{2}}g\|_{H^s}\leq C\|g\|_{H^s}$ for any $s\geq 0$.

Invoking the following calculus inequality which is just a consequence of Leibniz's formula,
\bbal
\|[D^{m},\mathbf{h}]\mathbf{f}\|_{L^2}\leq C(\|\na \mathbf{h}\|_{L^\infty}+\|\na^{|m|} \mathbf{h}\|_{L^\infty})\|\mathbf{f}\|_{H^{|m|-1}}, \ \text{for any} \ m=(m_1,m_2), \ m_1, m_2 \in \mathbb{N}^+,
\end{align*}
we obtain
\begin{equation}\label{I2}
\begin{split}
|I_2|\leq&~\sum_{|\beta|=1,2,3}\|[D^{\beta},U\cd] \na g\|_{L^2}\|\na g\|_{H^2}
\\
\leq&~C\Big(\|\na U\|_{L^\infty}+\|\na^3 U\|_{L^\infty}\Big)\|g\|^2_{H^3}
\\
\leq &~ C\|U\|_{L^\infty}\|g\|^2_{H^3},
\end{split}
\end{equation}
where in the last line we use the Bernstein's inequality(for $\mathrm{supp} \ \hat{U}(\xi)\subset\mathcal{C}$).
By Leibniz's formula and H\"{o}lder's inequality, one has
\begin{equation}\label{I3}
\begin{split}
|I_3|\leq&~C\|v\cd \na \Theta\|_{H^3}\|g\|_{H^3}
\\
\leq&~ C\Big(\|\na \Theta\|_{L^\infty}+\|\na^4  \Theta\|_{L^\infty}\Big)\|v\|_{H^3}\|g\|_{H^3}
\\
\leq&~ C\|\Theta\|_{L^\infty}\|g\|^2_{H^3}.
\end{split}
\end{equation}
Using  H\"{o}lder's inequality and Young inequality, we prove that
\begin{equation}\label{I}
\begin{split}
|I_{4}|&\leq~ C\|U\cd\na\Theta\|_{H^3}\|g\|_{H^3}
\\&\leq~ C\| U\cd\na\Theta\|_{H^3}+C\| U\cd\na\Theta\|_{H^3}\|g\|^2_{H^3},
\end{split}
\end{equation}
where in the last line we use the Bernstein's inequality.

Adding \eqref{10} and \eqref{l}, and putting all the estimates \eqref{I1}-\eqref{I} into \eqref{l}, we obtain
\bbal
\frac{\dd}{\dd t}\|g\|^2_{H^{3}}+\mu\|\Lambda^\alpha g\|^2_{H^{3}}&\leq C\left( \|\Lambda^\alpha g\|^2_{H^{3}}\|g\|_{H^3}+\|U\cd\na \Theta\|_{H^3} \right)
\\&\quad +C \|g\|^2_{H^3}
\left(\|U,\Theta\|_{L^\infty}+\|U\cd\na\Theta\|_{H^3}\right).
\end{align*}
Thus, we have
\bbal
\frac{\dd}{\dd t}\|g\|^2_{H^{3}}+(\mu-C\|g\|_{H^3})\|\Lambda^\alpha g\|^2_{H^{3}} &\leq C \|g\|^2_{H^3}
\left(\|U,\Theta\|_{L^\infty}+\|U\cd\na\Theta\|_{H^3}\right)
\\
& \quad + C\|U\cd\na \Theta\|_{H^3}.
\end{align*}
Assume $\|g(t)\|^2_{H^{3}}\leq (C+1)\epsilon$ and $\epsilon <  \frac{\mu^2}{4C^3}$. Then
 \begin{equation*}
   \mu-C\|g\|_{H^3}> \frac{\mu}{2}.
 \end{equation*}
 Using Gronwall's inequality and Condition \eqref{condition}, we conclude that
\begin{equation*}
\begin{split}
&\|g(t)\|^2_{H^3}+ \frac{\mu}{2}\int^{t}_0 \| g(\tau,\cdot)\|^2_{H^3}d\tau
\\
\leq & C\left(\|g_0\|^2_{H^3}+\int^{\infty}_{0}\|U\cd\na\Theta\|_{H^3}\dd t\right)\cd\exp\left(C\big(\int^{\infty}_{0}\|U,\Theta\|_{L^\infty}\dd t+\int^{\infty}_{0}\|U\cd\na\Theta\|_{H^3}\dd t\big)\right)
\\
 \leq & C\epsilon.
\end{split}
\end{equation*}
By bootstrap arguments, we conclude that
\begin{equation*}
  \|g(t)\|^2_{H^3}+ \frac{\mu}{2}\int^{t}_0 \| \Lambda^{\alpha}g(\tau,\cdot)\|^2_{H^3}d\tau \leq (C+1)\epsilon, \quad t \in [0, +\infty).
\end{equation*}
Thus, we complete the proof of Theorem \ref{thm}.
\end{proof}
\section{The proof of Corollary \ref{cc}}
In this section, we will give a proof of Corollary \ref{cc} by direct calculation.
\begin{proof}[Proof of Corollary \ref{cc}]
	By direct calculation, we get
	\begin{equation}\label{91}
	\begin{split}
	\|\hat{a}_0\|_{L^1_\xi } &= \int_{\mathbb{R}^2} \delta^{-1} \left( \log |\log \delta| \right) \cdot \hat{\chi}(\xi)d\xi
	\\
	&\geq \pi \big[(\frac{3}{2})^2-(\frac{4}{3})^2 \big] \cdot \frac{1}{3}\delta \cdot \delta^{-1} \left( \log |\log \delta| \right)
	\\
	& \geq \frac{1}{100}\log |\log \delta|,
	\end{split}
	\end{equation}
	and
	\begin{equation}\label{92}
	\begin{split}
	\|\hat{a}_0\|_{L^2_\xi } &= \delta^{-1} \left( \log |\log \delta| \right)  \left( \int_{\mathbb{R}^2} \hat{\chi}^2(\xi)d\xi \right)^{\frac{1}{2}}
	\\
	&\geq \delta^{-1}  \log |\log \delta| \cdot \left( \pi \big( (\frac{3}{2})^2-(\frac{4}{3})^2 \big) \cdot \frac{1}{3}\delta  \right)^{\frac{1}{2}}
	\\
	& \geq \frac{1}{100}\delta^{- \frac{1}{2}}\log |\log \delta|.
	\end{split}
	\end{equation}
	If $|x| \leq \frac{1}{200}$, we could derive that
	\begin{equation}\label{93}
	\begin{split}
	\|a_0\|_{L^\infty(|x| \leq \frac{1}{200})}&= \|\frac{1}{(2\pi)^2} \int_{\mathbb{R}^2} e^{ix\cdot \xi} \hat{a}_0(\xi)d\xi\|_{L^\infty(|x| \leq \frac{1}{200})}
	\\
	& \geq \frac{1}{500} \| \int_{\mathbb{R}^2}\cos(x \cdot \xi)\hat{a}_0(\xi)d\xi \|_{L^\infty(|x| \leq \frac{1}{200})}
	\\
	& \geq \frac{1}{5000}\log |\log \delta|.
	\end{split}
	\end{equation}
	If $|x| \in [\frac{3}{8}\pi, \frac{1}{3}\pi]$, we could prove that
	\begin{equation}\label{94}
	\begin{split}
	\|\partial_2a_0\|_{L^\infty(|x| \in [\frac{3}{8}\pi, \frac{1}{3}\pi])}&= \|\frac{1}{(2\pi)^2} \int_{\mathbb{R}^2} i \xi_2 e^{ix\cdot \xi} \hat{a}_0(\xi)d\xi\|_{L^\infty(|x| \in [\frac{3}{8}\pi, \frac{1}{3}\pi])}
	\\
	& \geq \frac{1}{100}\delta^{-1}\log |\log \delta| \cdot \int_{\mathbb{R}^2} | \sin(x\cdot \xi)\xi_2 | \cdot \hat{\chi}(\xi)d\xi
	\\
	& \geq \frac{1}{5000}\log |\log \delta|,
	\end{split}
	\end{equation}
	where we use the face that $\xi, -\xi$ appear in pairs in the support of $\hat{\chi}$ and $\sin(x\cdot \xi)\xi_2=\sin(x\cdot (-\xi))(-\xi_2)$.
	Recalling the expression of $U_0, \Theta_0$, we then conclude the proof of \eqref{c9} by using the above estimates \eqref{91}-\eqref{94}.

	By using
	\begin{equation*}
	\begin{split}
	\|\Theta\|_{L^\infty} &\leq \|\widehat{\Theta}(t,\xi)\|_{L^1_\xi}
	\leq Ce^{-\mu (\frac{4}{3})^\alpha t}\|\widehat{a}_0\|_{L^1_\xi},
	%\int^{\infty}_{0}||U\cd\na\Theta||_{H^3}\dd t
	\\
	\|U\|_{L^\infty} &\leq \|\widehat{U}(t,\xi)\|_{L^1_\xi} \leq \||\xi|^{-1}\xi^{\bot} \cdot \widehat{\Theta}(t,\xi)\|_{L^1_\xi}
	\leq Ce^{-\mu (\frac{4}{3})^\alpha t}\|\widehat{a}_0\|_{L^1_\xi},
	\end{split}
	\end{equation*}
	we get
	$$ \int^{\infty}_{0}\|U, \Theta\|_{L^\infty}\dd t \leq C  \log |\log \delta|.$$
Notice that
	\begin{equation*}
	\begin{split}
	%			&(U \cdot \nabla) U^i=(U^1+U_2) \partial_1U^i+U^2(\partial_2-\partial_1)U^i, \ i=1,2,
	%			\\
	&(U \cdot \nabla) \Theta=(U^1+U^2) \partial_1 \Theta+U^2(\partial_2-\partial_1)\Theta.
	\end{split}
	\end{equation*}
	Thus, we have
	\begin{equation*}
	\begin{split}
	\|(U \cdot \nabla) \Theta\|_{H^3}
	\leq &C\left(\| U^1 + U^2\|_{L^\infty} \|\Theta\|_{H^4}+ \| U^1 + U^2\|_{H^3} \|\nabla \Theta\|_{L^\infty} \right)
	\\
	& +
	C\left(\|U\|_{L^\infty} \|(\partial_2-\partial_1)\Theta\|_{H^3}+ \| U\|_{H^3} \|(\partial_2-\partial_1)\Theta\|_{L^\infty} \right)
	\\
	& \leq C \delta e^{-2\mu (\frac{3}{4})^\alpha t}\|\widehat{a}_0\|_{L^1_\xi} \|a_0\|_{L^2},
	%(U \cdot \nabla) \Theta=(U^1+U_2) \partial_1 \Theta+U^2(\partial_2-\partial_1)\Theta.
	\end{split}
	\end{equation*}
	where we use the fact that $U^1+U^2=-(-\Delta)^{-\frac{1}{2}}(\partial_2-\partial_1)\Theta$. We then derive that
	\begin{equation*}
	\begin{split}
	&\int^{\infty}_{0}\|U\cd\na\Theta\|_{H^3}\dd t \leq C \delta^{\frac{1}{2}} (\log |\log \delta|)^2,
	\end{split}
	\end{equation*}
In a result, we conclude that
	\begin{equation*}
	\begin{split}
	&\left(\|g_0\|^2_{H^3}+\int^{\infty}_{0}\|U\cd\na\Theta\|_{H^3}\dd t\right)\cd\exp\left(C\big(\int^{\infty}_{0}\|\Theta\|_{L^\infty}\dd t+\int^{\infty}_{0}\|U\cd\na\Theta\|_{H^3}\dd t\big)\right)
	\\
	& \leq C (\delta^{\frac{1}{2}} \cdot |\log \delta| + \delta^{\frac{1}{2}}  (\log |\log \delta|) \cdot |\log \delta| )
	\\
	& \leq \delta^{\frac{1}{2}} |\log \delta|^2= \epsilon.
	\end{split}
	\end{equation*}
		Therefore, we complete the proof of Corollary \ref{cc}.
\end{proof}
\section*{Acknowledgments} The authors would like to express thanks to the reviewers for their helpful advice. The author Huali Zhang also would like to thank Professor.Dongbing Zha for letting her know the reference \cite{Liu}. The work of Huali Zhang is supported by the State Scholarship Fund of China Scholarship Council (No. 201808430121) and Hunan Provincial Key Laboratory of Intelligent Processing of Big Data on Transportation, Changsha University of Science and Technology, Changsha; 410114, China. The work of Jinlu Li is supported by the National
Natural Science Foundation of China (Grant No. 11801090) and Postdoctoral Science Foundation of
China (2020T130129 and 2020M672565).

\section*{Conflicts of Interest}
The authors declared that this work does not have any conflicts of interest.
 
\end{document}